\documentclass[10pt,leqno]{amsart}
\theoremstyle{definition}
\usepackage{indentfirst, amssymb,amsmath,amsthm}
\usepackage{csquotes}
\usepackage{hyperref}
\newtheorem*{theoA}{Theorem A}
\newtheorem*{theoB}{Theorem B}
\newtheorem*{theoC}{Theorem C}
\newtheorem*{theoD}{Theorem D}
\newtheorem*{theoE}{Theorem E}
\newtheorem*{theoF}{Theorem F}

\newtheorem{theo}{Theorem}[section]
\newtheorem{lem}{Lemma}[section]
\newtheorem{cor}{Corollary}[section]

\newtheorem{exm}{Example}[section]
\newtheorem{defi}{Definition}[section]
\newtheorem{rem}{Remark}[section]
\newtheorem{question}{Question}[section]
\newcommand{\ol}{\overline}
\newcommand{\be}{\begin{equation}}
\newcommand{\ee}{\end{equation}}
\newcommand{\beas}{\begin{eqnarray*}}
\newcommand{\eeas}{\end{eqnarray*}}
\newcommand{\bea}{\begin{eqnarray}}
\newcommand{\eea}{\end{eqnarray}}

\numberwithin{equation}{section}
\AtBeginDocument{{\noindent\small Mathematica Moravica, Vol. 20-2 (2016), 1-14\\
DOI: 10.5937/MatMor1602001A}}
\begin{document}
\title[Uniqueness of the power of a meromorphic functions]{Uniqueness of the power of a meromorphic functions with its differential polynomial\\ sharing a set}
\date{}
\author[A. Banerjee and B. Chakraborty]{ Abhijit Banerjee  and Bikash Chakraborty. }
\date{}
\address{Department of Mathematics, University of Kalyani, West Bengal 741235, India.}
\email{abanerjee\_kal@yahoo.co.in, abanerjee\_kal@rediffmail.com
}
\email{bikashchakraborty.math@yahoo.com, bchakraborty@klyuniv.ac.in}
\maketitle
\let\thefootnote\relax
\footnotetext{Type set by \AmS -\LaTeX}
\setcounter{footnote}{0}
\begin{abstract}
This paper is devoted to the uniqueness problem of the power of a meromorphic function with its differential polynomial sharing a set. Our result will extend a number of results obtained in the theory of normal families. Some questions are posed for future research.
\end{abstract}
\section{Introduction Definitions and Results}
In this paper we assume that readers are familiar with the basic Nevanlinna Theory (\cite{5}). Let  $f$ and $g$ be two  non constant meromorphic functions in the complex plane $\mathbb{C}$. If for some $a\in\mathbb{C}\cup\{\infty\}$, $f$ and $g$ have same set of $a$-points with the same multiplicities, we say that $f$ and $g$ share the value $a$ CM (counting multiplicities) and if we do not consider the multiplicities then $f$, $g$ are said to share the value $a$ IM (ignoring multiplicities).\par
When $a=\infty$ the zeros of $f-a$ means the poles of $f$.\par
The problem of meromorphic functions sharing values with their derivatives is a special subclass in the literature of uniqueness theory.
The subject of sharing values between entire functions and their derivatives was first studied by Rubel and Yang (\cite{9}).
In 1977, they proved that if a non-constant entire function $f$ and $f^{'}$ share two distinct finite numbers $a$, $b$ CM, then $f = f^{'}$.\par
In 1979, analogous result for IM sharing was obtained by Mues and Steinmetz in the following manner.
\begin{theoA}(\cite{8}) Let $f$ be a non-constant entire function. If $f$ and $f^{'}$ share two distinct values $a$, $b$ IM then $f^{'}\equiv f$.\end{theoA}
To proceed further we consider the following well known definition of set sharing. \par
Let $S$ be a set of complex numbers and $E_{f}(S)=\bigcup_{a\in S}\{z: f(z)=a\}$, where each zero is counted according to its multiplicity. If we do not count the multiplicity, then the set $ \bigcup_{a\in S}\{z: f(z)=a\}$ is denoted by $\ol E_{f}(S)$.\par
If $E_{f}(S)=E_{g}(S)$ we say that $f$ and $g$ share the set $S$ CM. On the other hand, if $\ol E_{f}(S)=\ol E_{g}(S)$, we say that $f$ and $g$ share the set $S$ IM. Evidently, if $S$ contains only one element, then it coincides with the usual definition of CM (respectively, IM) sharing of values.\par
In view of the above definition it will be interesting to study the relation between $f$ and its derivative $f^{'}$ when they share a set.
We see from the following example that results of Rubel-Yang or Mues-Steinmetz are not in general true when we consider the sharing of a set of two elements instead of values.
\begin{exm} Let $S=\{a,b\}$, where $a$ and $b$ are any two distinct complex numbers. Let $f(z)=e^{-z}+a+b$, then $E_{f}(S)=E_{f'}(S)$ but $f\not\equiv f'$.\end{exm}
Thus for the uniqueness of meromorphic function with its derivative counterpart, the cardinality of the sharing set should at least be three. In this direction, in 2003, using Normal families, Fang and Zalcman made the first breakthrough by establishing the following result.
\begin{theoB}(\cite{3}) Let $S=\{0,a,b\}$, where $a,b$ are two non-zero distinct complex numbers satisfying $a^{2} \not= b^{2}$, $a\not = 2b$, $a^{2}-ab+b^{2}\not=0$. If for a non constant entire function $f$,  $E_{f}(S)=E_{f'}(S)$, then $f\equiv f'$.
\end{theoB}
In 2007 Chang, Fang and Zalcman(\cite{2.0}) further extended the above result by considering an arbitrary set having three elements in the following manner.
\begin{theoC}(\cite{2.0}) Let $f$ be a non-constant entire function and let $S=\{a,b,c\}$, where $a,b~and~c$ are distinct complex numbers. If $E_{f}(S)=E_{f'}(S)$, then either
\begin{enumerate}
\item $f(z)=Ce^{z}$; or
\item $f(z)=Ce^{-z}+\frac{2}{3}(a+b+c)$ and $(2a-b-c)(2b-c-a)(2c-a-b)=0$; or
\item $f(z)=Ce^{\frac{-1 \pm i\sqrt{3}}{2}z}+\frac{3 \pm i\sqrt{3}}{6}(a+b+c)$ and $a^2+b^2+c^2-ab-bc-ca=0$,\\
where $C$ is a non-zero constant.
\end{enumerate}
\end{theoC}
In the next year, Chang and Zalcman(\cite{2}) replaced the entire function by meromorphic function with at most finitely many simple poles in Theorem B and C and obtained similar results as follows.
\begin{theoD}(\cite{2}) Let $S=\{0,a,b\}$, where $a,b$ are two non-zero distinct complex numbers. If $f$ is a meromorphic function with at most finitely many poles and $E_{f}(S)=E_{f'}(S)$, then $f\equiv f'$.
\end{theoD}
\begin{theoE}(\cite{2}) Let $f$ be a non-constant meromorphic function with at most finitely many simple poles; and let $S=\{0,a,b\}$, where $a,b$ are distinct non zero complex numbers. If $E_{f}(S)=E_{f'}(S)$, then either
\begin{enumerate}
\item $f(z)=Ce^{z}$; or
\item $f(z)=Ce^{-z}+\frac{2}{3}(a+b)$ and either $(a+b)=0$ or $(2a^{2}-5ab+2b^{2})=0$; or
\item $f(z)=Ce^{\frac{-1 \pm i\sqrt{3}}{2}z}+\frac{3 \pm i\sqrt{3}}{6}(a+b)$ and $a^2-ab+b^2=0$,\\
where $C$ is a non-zero constant.
\end{enumerate}
\end{theoE}
In 2011, Feng L\"{u}(\cite{6.0}) consider an arbitrary set with three elements in Theorem E and got the same result with some additional suppositions.
He obtained the following result.

\begin{theoF}(\cite{6.0}) Let $f$ be a non-constant transcendental meromorphic function with at most finitely many simple poles; and let $S=\{a,b,c\}$, where $a,b,~and~c$ are distinct complex numbers. If $E_{f}(S)=E_{f'}(S)$, then either
\begin{enumerate}
\item $f(z)=Ce^{z}$; or
\item $f(z)=Ce^{-z}+\frac{2}{3}(a+b+c)$ and $(2a-b-c)(2b-c-a)(2c-a-b)=0$; or
\item $f(z)=Ce^{\frac{-1 \pm i\sqrt{3}}{2}z}+\frac{3 \pm i\sqrt{3}}{6}(a+b+c)$ and $a^2+b^2+c^2-ab-bc-ca=0$,\\
where $C$ is a non-zero constant.
\end{enumerate}
\end{theoF}
So we observe from the above mentioned results that the researchers were mainly involved to find the uniqueness of an entire or meromorphic function with its first derivative sharing a set at the expanse of allowing several constraints.
But all were practically tacit about the uniqueness of an entire or meromorphic function with its higher order derivatives.
 In 2007  Chang, Fang and Zalcman (\cite{2.0}) consider the following example to show that in Theorem C, one can not relax the CM sharing to IM sharing of the set $S$. In other words, when multiplicity is disregarded, the uniqueness result ceases to hold.
\begin {exm} Let $S=\{-1,0,1\}$ and $f(z)=\sin z$. Then $f$ and $f'$ share $S$ IM but $f\not\equiv f'$. \end{exm}
  Thus it is natural to ask the following question :-
\begin{ques} Does there exist any set which when shared by a meromorphic function together with its higher order derivative or even a power of a meromorphic function together with its differential polynomial, lead to wards the uniqueness ?\end{ques}
To seek the possible answer of the above question is the motivation of the paper. We answer the above question even under relaxed sharing hypothesis.
 To this end, we resort to the notion of weighted sharing of sets appeared in the literature in 2001 (\cite{6}).
\begin{defi} (\cite{6}) Let $k$ be a nonnegative integer or infinity. For $a\in\mathbb{C}\cup\{\infty\}$ we denote by $E_{k}(a;f)$ the set of all $a$-points of $f$, where an $a$-point of multiplicity $m$ is counted $m$ times if $m\leq k$ and $k+1$ times if $m>k$. If $E_{k}(a;f)=E_{k}(a;g)$, we say that $f,g$ share the value $a$ with weight $k$.\end{defi}

We write $f$, $g$ share $(a,k)$ to mean that $f$, $g$ share the value $a$ with weight $k$. Clearly if $f$, $g$ share $(a,k)$, then $f$, $g$ share $(a,p)$ for any integer $p$, $0\leq p<k$. Also we note that $f$, $g$ share a value $a$ IM or CM if and only if $f$, $g$ share $(a,0)$ or $(a,\infty)$ respectively.\par
\begin{defi}(\cite{6}) Let $S$ be a set of distinct elements of $\mathbb{C}\cup\{\infty\}$ and $k$ be a nonnegative integer or $\infty$. We denote by $E_{f}(S,k)$ the set $\bigcup_{a\in S}E_{k}(a;f)$. If $E_{f}(S,k)=E_{g}(S,k)$, then we say $f$, $g$ share the set $S$ with weight $k$. \end{defi}
Throughout the paper we use the following notation for $L$.
\begin{defi}
Let $k(\geq1),l(\geq1)$ be positive integers and $a_{i} \in \mathbb{C}$ for $i=0,2,...,k-1$. For a non constant meromorphic function $f$, we define the differential polynomial in $f$ as  $$L=L(f)=a_{0}(f^{(k)})^{l}+a_{1}(f^{(k-1)})^{l}+...+a_{k-1}(f^{'})^{l}.$$
\end{defi}
First suppose $P(z)$ is defined by \be\label{e5.1} P(z)=az^{n}-n(n-1)z^{2}+2n(n-2)bz-(n-1)(n-2)b^{2}\ee where $n\geq 3$ is an integer and $a$ and $b$ are two nonzero complex numbers satisfying $ab^{n-2}\not=2$. We have from (\ref{e5.1}) \bea\label{e5.2}P^{'}(z)&=& naz^{n-1}-2n(n-1)z+2n(n-2)b\\&=&\frac{n}{z}[az^{n}-2(n-1)z^{2}+2(n-2)bz].\nonumber\eea
We note that $P^{'}(0)\not= 0$ and so from (\ref{e5.1}) $P^{'}(z)=0$ implies
$$az^{n}-2(n-1)z^{2}+2(n-2)bz=0.$$
Now at each root of $P^{'}(z)=0$ we get
\beas &&P(z)\\&=&az^{n}-n(n-1)z^{2}+2n(n-2)bz-(n-1)(n-2)b^{2}\\&=& 2(n-1)z^{2}-2(n-2)bz-n(n-1)z^{2}+2n(n-2)bz-(n-1)(n-2)b^{2}\\& =& -(n-1)(n-2)(z-b)^{2}\eeas
So at a root of $P^{'}(z)=0$, $P(z)$ will be zero if $P^{'}(b)=0$. But $P^{'}(b)=nb(ab^{n-2}-2)\not =0$, which implies that a zero of $P^{'}(z)$ is not a zero of $P(z)$. In other words each zero of $P(z)$ is simple.
The following theorem is the main result of this paper which answers Question 1.1.
\begin{theo}\label{thB1} Let $m(\geq1),n(\geq1)$ be positive integers and $f$ be a non constant meromorphic function.
Suppose that $S=\{z :P(z)=0 \}$ and $E_{f^{m}}(S,p)=E_{L(f)}(S,p)$. If one of the following conditions holds:
\begin{enumerate}
\item $2 \leq p < \infty$ and $n > 6+6\frac{\mu+1}{\lambda-2\mu},$
\item $p=1$ and $n > \frac{13}{2}+7\frac{\mu+1}{\lambda-2\mu},$
\item $p=0$ and $n> 6+3\mu+6\frac{(\mu+1)^{2}}{\lambda-2\mu};$
\end{enumerate}
then  $f^{m} \equiv L(f)$, where $\lambda=\min\{m(n-2)-1,(1+k)l(n-2)-1\}$ and $\mu=\min\{\frac{1}{p},1\}$.
\end{theo}
\begin{cor}\label{thB12}
There exists a set $S$ with eight(seven) elements such that if a non constant meromorphic (entire) function $f$ and its k-th derivative $f^{(k)}$ satisfy $E_{f}(S,2)=E_{f^{(k)}}(S,2)$, then $f\equiv f^{(k)}$.
\end{cor}
The following example shows that for a non-constant entire function the set $S$ in Theorem 1.1 can not be replaced by an arbitrary set containing seven distinct elements.
\begin{exm} For a non-zero complex number $a$, let $S=\{0,a\omega, a\sqrt{\omega},a, \frac{a}{\sqrt{\omega}}, \frac{a}{\omega}, \frac{a}{\omega \sqrt{\omega}}\}$, where $\omega$ is the non-real cubic root of unity. Choosing $f=e^{{\omega}^{^{\frac{1}{2k}}}z}$, it is easy to verify that $f$ and $f^{(k)}$ share $(S,\infty)$, but $f\not \equiv f^{(k)}$ \end{exm}
\begin{rem} However the following questions are still open.
\begin{enumerate}
\item Can the cardinality of the set $S$ be further reduced in the Theorem \ref{thB1} and specially in Corollary \ref{thB12} without imposing any constraints on the functions?
\item Can the conclusion of Theorem \ref{thB1} remain valid if any non-homogeneous differential polynomial generated by $f$ is considered?
\end{enumerate}
\end{rem}
\section {Lemmmas}
We define $R(z)=\frac{az^{n}}{n(n-1)(z-\alpha_{1})(z-\alpha_{2})}$, where $\alpha_{1}$ and $\alpha_{2}$ are the distinct roots of the equation $$n(n-1)z^{2}-2n(n-2)bz+(n-1)(n-2)b^{2}=0.$$
Now let  $F=R(f^{m})$ , $G=R(L(f))$ and
$$H=(\frac{F''}{F'}-\frac{2F'}{F-1})-(\frac{G''}{G'}-\frac{2G'}{G-1}).$$
\begin{lem}
For any two non-constant meromorphic functions $f_{1}$ and $f_{2}$,
$$\ol{N}(r,f_{1}f_{2}) \leq \ol{N}(r,f_{1})+\ol{N}(r,f_{2}).$$
\end{lem}
\begin{lem}\label{l.n.1}
 Let $F$ and $G$ share $(1,p)$ where $F$ and $G$ defined as earlier, then\par
$\overline{N}_{L}(r,1;F)\leq \mu(\overline{N}(r,0;f)+\overline{N}(r,\infty;f))+S(r,f)$ \\
$\overline{N}_{L}(r,1;G)\leq \mu (\overline{N}(r,0; L(f))+\overline{N}(r,\infty;f))+S(r,f)$\\
where $\mu=\min\{\frac{1}{p},1\}.$
\end{lem}
\begin{proof}
When $p=0$ we get
\beas \overline{N}_{L}(r,1;F) &\leq&  N(r,1;F)-\ol{N}(r,1,F)\\
&\leq&  N(r,\frac{(f^{m})'}{f^{m}})+S(r,f)\\
&\leq& (\overline{N}(r,0;f^{m})+\overline{N}(r,\infty;f^{m}))+S(r,f)\\
&\leq& (\overline{N}(r,0;f)+\overline{N}(r,\infty;f))+S(r,f)\eeas
When $p\geq 1$ we get
\beas \overline{N}_{L}(r,1;F) &\leq& \overline{N}(r,1;F|\geq p+1)\\
&\leq& \frac{1}{p} (N(r,1;F)-\ol{N}(r,1,F))\\
&\leq& \frac{1}{p} N(r,\frac{(f^{m})'}{(f^{m})})+S(r,f)\\
&\leq& \frac{1}{p}(\overline{N}(r,0;f^{m})+\overline{N}(r,\infty;f^{m}))+S(r,f)\\
&\leq& \frac{1}{p}(\overline{N}(r,0;f)+\overline{N}(r,\infty;f))+S(r,f)\eeas
Combining the two cases we get the proof.
\end{proof}
\begin{lem}\label{l1.1.1} Let $F$ and $G$ share $(1,p)$ where $F$ and $G$ defined as earlier. If $F \not\equiv G$ then
\bea\label{el.1}
\overline{N}(r,f) &\leq& \frac{\mu+1}{\lambda-2\mu}(\overline{N}(r,0;f)+\overline{N}(r,0;L(f)))+S(r,f),\eea
where $\lambda=\min\{m(n-2)-1,(1+k)l(n-2)-1\}$ and $\mu=\min\{\frac{1}{p},1\}$.
\end{lem}
\begin{proof} Let us define $V=(\frac{F'}{F(F-1)}-\frac{G'}{G(G-1)})$\\
\textbf{Case-1} $V\equiv 0$\\
By integration we get $(1-\frac{1}{F})=A(1-\frac{1}{G})$.
As $f^{m}$ and $L(f)$ share $(\infty,0)$, so if $\overline{N}(r,f)\neq S(r,f)$ then $A=1$, i.e., $F=G$, which is not possible. So $\overline{N}(r,f)=S(r,f).$ Thus the lemma holds.\\
\textbf{Case-2} $V\not\equiv 0$\\
Let $z_{0}$ be a pole of $f$ of order $t$, then it is a pole of $L(f)$ of order $(t+k)l$ and that of $F$ and $G$ are $tm(n-2)$ and $(t+k)l(n-2)$ respectively.\\
Clearly $z_{0}$ is a zero of $(\frac{F'}{F-1}-\frac{F'}{F})$ order at least $tm(n-2)-1$ and zero of $V$ of order atleast $\lambda$, where $\lambda=\min\{m(n-2)-1,(1+k)l(n-2)-1\}$\\
Thus \beas &&\overline{N}(r,\infty;f)\\ &\leq& \frac{1}{\lambda}N(r,0;V)\\
&\leq& \frac{1}{\lambda}N(r,\infty;V)+S(r,f)\\
&\leq& \frac{1}{\lambda}\{\overline{N}_{L}(r,1;F)+\overline{N}_{L}(r,1;G)+\overline{N}(r,0;f)+\overline{N}(r,0;L(f))\}+S(r,f)\\
&\leq& \frac{1}{\lambda}[\mu\{\overline{N}(r,0;f)+\overline{N}(r,\infty;f)+\overline{N}(r,0; L(f))+\overline{N}(r,\infty;f)\}\\
&+& \overline{N}(r,0;f)+\overline{N}(r,0;L(f))]+S(r,f).
\eeas
Thus
\beas \overline{N}(r,\infty;f) &\leq& \frac{\mu+1}{\lambda-2\mu}(\overline{N}(r,0;f)+\overline{N}(r,0;L(f)))+S(r,f).\eeas
\end{proof}

\begin{lem}\label{el.2.1.2} If $H \not\equiv0$ and $F$ and $G$ share $(1,p)$ then
\bea\label{el.2} &&N(r,\infty;H)\\ &\leq& \overline{N}(r,\infty;f)+\overline{N}(r,0;f)+\overline{N}(r,0;L)+\overline{N}(r,b;f^{m})+\overline{N}(r,b;L)\nonumber\\
\nonumber&+&\overline{N}_{L}(r,1;F)+\overline{N}_{L}(r,1;G)+\overline{N}_{0}(r,0;(f^{m})')+\overline{N}_{0}(r,0;L'),\eea
where $\overline{N}_{0}(r,0;(f^{m})')$ denotes the counting function of the zeros of $(f^{m})'$ which are not the zeros of $f(f^{m}-b)$ and $F-1$, similarly $\overline{N}_{0}(r,0;L')$ is defined.
\end{lem}
\begin{proof} The proof is obvious if we are keeping the following in our mind :\\
$\overline{N}(r,\infty;F)\leq\overline{N}(r,\infty;f)+\overline{N}(r,\alpha_{1};f^{m})+\overline{N}(r,\alpha_{2};f^{m})$,\\
But simple zeros of $f^{m}-\alpha_{i}$ are not poles of $H$ and multiple zeros of  $f^{m}-\alpha_{i}$ are zeros of $(f^{m})'$. Similar explanation for $G$ is also hold.
\end{proof}
\begin{lem}(\cite{1})\label{lemma} Let $$Q(z)=(n-1)^{2}(z^{n}-1)(z^{n-2}-1)-n(n-2)(z^{n-1}-1)^{2}$$
then $$Q(z)=(z-1)^{4}\prod\limits_{i=1}^{2n-6}(z-\beta_{i})$$
where $\beta_{i} \in \mathbb{C}\setminus\{0,1\} (i=1,2,\ldots,2n-6),$ which are distinct.
\end{lem}

\section {Proof of the theorem}

\begin{proof} [\textbf{Proof of Theorem\ref{thB1} }]
\textbf{Case-1} $H \not\equiv 0$\\
Then clearly $F \not\equiv G$.\\
Clearly $\overline{N}(r,1;F|=1)=\overline{N}(r,1;G|=1)\leq N(r,\infty;H)$\\
Now using the Second Fundamental Theorem and Lemma \ref{el.2.1.2},  we get\\
\bea\label{pe1.1.1} &&(n+1)T(r,f^{m})\\ &\leq& \overline{N}(r,\infty;f)+\overline{N}(r,0;f)+\overline{N}(r,b;f^{m})\nonumber\\
\nonumber &+& \overline{N}(r,1;F)-N_{0}(r,0,(f^{m})')+S(r,f)\\
\nonumber &\leq& 2\{\overline{N}(r,\infty;f)+\overline{N}(r,0;f)+\overline{N}(r,b;f^{m})\}\\
\nonumber &+& \{\overline{N}(r,0;L(f))+\overline{N}(r,b;L(f))\}+\{\overline{N}(r,1;F|\geq2)\\
\nonumber &+& \overline{N}_{L}(r,1;F)+\overline{N}_{L}(r,1;G)+\overline{N}_{0}(r,0;(L(f))')\}+S(r,f).
\eea
\textbf{Subcase-1.1} $p\geq2$\\
Now, \bea\label{pe1.1.2} &&\overline{N}(r,1;F|\geq2)+\overline{N}_{L}(r,1;F)+\overline{N}_{L}(r,1;G)+\overline{N}_{0}(r,0;(L(f))')\\
\nonumber &\leq& \overline{N}(r,1;G|\geq2)+\overline{N}(r,1;G|\geq3)+\overline{N}_{0}(r,0;(L(f))')\\
\nonumber &\leq& N(r,0;L'|L\neq0)+S(r,f)\\
\nonumber &\leq& N(r,\frac{L'}{L})+S(r,f)\\
\nonumber &\leq& \overline{N}(r,0;L)+\overline{N}(r,\infty;f)+S(r,f)\eea
Thus,
\bea\label{pe1.1.3} &&(n+1)T(r,f^{m})\\ &\leq& 2\{\overline{N}(r,\infty;f)+\overline{N}(r,0;f)+\overline{N}(r,b;f^{m})\}\nonumber\\
\nonumber &+& 2\overline{N}(r,0;L)+\overline{N}(r,b;L)+\overline{N}(r,\infty;f)+S(r,f).
\eea
Similar result we get for $L(f)$ as
\bea\label{pe1.1.4} &&(n+1)T(r,L(f))\\ &\leq& 2\{\overline{N}(r,\infty;f)+\overline{N}(r,0;L(f))+\overline{N}(r,b;L(f))\}\nonumber\\
\nonumber &+& 2\overline{N}(r,0;f)+\overline{N}(r,b;f^{m})+\overline{N}(r,\infty;f)+S(r,f).
\eea
Let $T(r)=T(r,f^{m})+T(r,L(f))$ and $S(r)=S(r,f)$\\
By adding inequalities (\ref{pe1.1.3}) and (\ref{pe1.1.4}), we get
\bea\label{pe1.1.5} (n+1)T(r) &\leq& 6\overline{N}(r,\infty;f)+4\{\overline{N}(r,0;f)+\overline{N}(r,0;L(f)\}\\
\nonumber &+& 3\{\overline{N}(r,b;f^{m})+\overline{N}(r,b;L(f))\}+S(r,f).
\eea
\bea\label{pe1.1.6} (n-6)T(r) &\leq& 6\overline{N}(r,\infty;f)+S(r).
\eea

By using Lemma \ref{l1.1.1} and inequality (\ref{pe1.1.6}), we get
\bea\nonumber \label{pe1.1.7} (n-6)T(r) &\leq& 6\frac{\mu+1}{\lambda-2\mu}(\overline{N}(r,0;f)+\overline{N}(r,0;L(f)))+S(r)\\
\nonumber  &\leq& 6\frac{\mu+1}{\lambda-2\mu} T(r)+S(r),\eea
which is a contradiction as $n> 6+6\frac{\mu+1}{\lambda-2\mu}$.\\
\textbf{Subcase-1.2} $p=1$\\
Now, \bea\label{pe1.1.7} &&\overline{N}(r,1;F|\geq2)+\overline{N}_{L}(r,1;F)+\overline{N}_{L}(r,1;G)+\overline{N}_{0}(r,0;L')\\
\nonumber &\leq& \overline{N}(r,1;G|\geq2)+\overline{N}(r,1;F|\geq2)+\overline{N}_{0}(r,0;L')\\
\nonumber &\leq& N(r,0;L'|L\neq0)+\frac{1}{2}N(r,0;(f^{m})'| f^{m}\not=0)+S(r,f)\\
\nonumber &\leq& \overline{N}(r,0;L)+\overline{N}(r,\infty;L)+\frac{1}{2}\{\overline{N}(r,0;f)+\overline{N}(r,\infty;f)\}+S(r,f)\eea
Thus we get from (\ref{pe1.1.1}),
\bea\label{pe1.1.8}&& (n+1)T(r,f^{m})\\ &\leq& \frac{5}{2}\{\overline{N}(r,\infty;f)+\overline{N}(r,0;f)\}+2\overline{N}(r,b;f^{m})\nonumber\\
\nonumber &+& 2\overline{N}(r,0;L(f))+\overline{N}(r,b;L(f))+\overline{N}(r,\infty;f)+S(r,f).
\eea
Similar result we get for $L(f)$ as
\bea\label{pe1.1.9} &&(n+1)T(r,L(f))\\ &\leq& \frac{5}{2}\{\overline{N}(r,\infty;f)+\overline{N}(r,0;L(f))\}+2\overline{N}(r,b;L(f))\nonumber\\
\nonumber &+& 2\overline{N}(r,0;f)+\overline{N}(r,b;f^{m})+\overline{N}(r,\infty;f)+S(r,f).
\eea
Adding (\ref{pe1.1.8}) and (\ref{pe1.1.9}), we get
\bea\label{pe1.1.10} (n+1)T(r) &\leq& 7\overline{N}(r,\infty;f)+\frac{9}{2}\{\overline{N}(r,0;f)+\overline{N}(r,0;L(f))\}\\
\nonumber &+& 3\{\overline{N}(r,b;f^{m})+\overline{N}(r,b;L(f))\}+S(r).
\eea
\bea\label{pe1.1.11} (n-\frac{13}{2})T(r) &\leq& 7\overline{N}(r,\infty;f)+S(r).
\eea
Using Lemma \ref{l1.1.1}, we get
\bea\nonumber \label{pe1.1.12} (n-\frac{13}{2})T(r) &\leq& 7\frac{\mu+1}{\lambda-2\mu}(\overline{N}(r,0;f)+\overline{N}(r,0;L(f)))+S(r)\\
\nonumber  &\leq& 7\frac{\mu+1}{\lambda-2\mu} T(r)+S(r),\eea
which is a contradiction as $n > \frac{13}{2}+7\frac{\mu+1}{\lambda-2\mu}$.\\
\textbf{Subcase-1.3} $p=0$\\
Now using the Second Fundamental Theorem and Lemma \ref{el.2.1.2}, we get\\
\bea\label{pe1.1.12} &&(n+1)T(r)\\ &\leq& \overline{N}(r,\infty;f^{m})+\overline{N}(r,\infty;L(f))+\overline{N}(r,0;f^{m})+\overline{N}(r,0;L(f))\nonumber\\
\nonumber &+& \overline{N}(r,b;f^{m})+\overline{N}(r,b;L(f))+\overline{N}(r,1;F)+\overline{N}(r,1;G)\\
\nonumber &-&N_{0}(r,0;(f^{m})')-N_{0}(r,0;L(f)')+S(r,f)\\
\nonumber&\leq& 3\overline{N}(r,\infty;f)+2\overline{N}(r,0;f)+2\overline{N}(r,0;L(f))\\
\nonumber &+& 2\overline{N}(r,b;f^{m})+2\overline{N}(r,b;L(f))+\overline{N}(r,1;F)+\overline{N}(r,1;G)\\
\nonumber &-&\overline{N}(r,1;F|=1)+\overline{N}_{L}(r,1;F)+\overline{N}_{L}(r,1;G)+S(r,f)
\eea
Again,
$$\overline{N}(r,1;F)+\overline{N}(r,1;G)-\overline{N}(r,1;F|=1)\leq \overline{N}_{L}(r,1;F)+N(r,1;G)$$
i.e.,
\beas&&\overline{N}(r,1;F)+\overline{N}(r,1;G)-\overline{N}(r,1;F|=1)\\&\leq& \frac{1}{2}(\overline{N}_{L}(r,1;F)+\overline{N}_{L}(r,1;G)+N(r,1;G)+N(r,1;F))\eeas
So, with the help of Lemma \ref{l.n.1} and  Lemma \ref{l1.1.1} (\ref{pe1.1.12}) becomes  \\
\bea\label{pe1.1.13}&& (n+1)T(r)\\&\leq& 3\overline{N}(r,\infty;f)+2\overline{N}(r,0;f)+2\overline{N}(r,0;L(f))\nonumber\\
\nonumber &+& 2\overline{N}(r,b;f^{m})+2\overline{N}(r,b;L(f))+\frac{3}{2}(\overline{N}_{L}(r,1;F)+\overline{N}_{L}(r,1;G))\\
\nonumber &+& \frac{1}{2}(N(r,1;F)+N(r,1;G))+S(r,f)\eea
That is,
\beas &&(n-6)T(r) \\&\leq& 6\overline{N}(r,\infty;f)+3(\overline{N}_{L}(r,1;F)+\overline{N}_{L}(r,1;G))+S(r)\nonumber\\
&\leq& 6\overline{N}(r,\infty;f)+3 \mu (\overline{N}(r,0; L(f))+\overline{N}(r,0;f)+2\overline{N}(r,\infty;f))+S(r)
\eeas
Thus
\beas
(n-6-3\mu)T(r) &\leq& 6\frac{(\mu+1)^{2}}{\lambda-2\mu}(\overline{N}(r,0;f)+\overline{N}(r,0;L(f)))+S(r)\\
 &\leq& 6\frac{(\mu+1)^{2}}{\lambda-2\mu}T(r)+S(r),
\eeas
which is a contradiction as $n> 6+3\mu+6\frac{(\mu+1)^{2}}{\lambda-2\mu}$.\\
\textbf{Case-2} $H\equiv 0$\\
In this case $F$ and $G$ share $(1,\infty)$.\par
Now by integration we have
\bea\label{pe1.1} F=\frac{AG+B}{CG+D},\eea
where $A,B,C,D$ are constant satisfying $AD-BC\neq 0 $.\par
Thus by Mokhon'ko's Lemma (\cite{7}) \bea\label{pe1.2} T(r,f^{m})=T(r,L(f))+S(r,f)\eea
Clearly from equation (\ref{pe1.1}) when $n\geq3$ we get $\overline{N}(r,f)=S(r,f)$ if $C\neq0$, otherwise when $C=0$, $f^{m}$ and $L(f)$ share $(\infty,\infty)$.\par
As $AD-BC\neq0$, so $A=C=0$ never occur. Thus we consider the following cases:\\
\textbf{Subcase-2.1} $AC\neq0$
In this case \bea F-\frac{A}{C}=\frac{BC-AD}{C(CG+D)}.\eea
So, $$\overline{N}(r,\frac{A}{C};F)=\overline{N}(r,G).$$
Now using the Second Fundamental Theorem and (\ref{pe1.2}), we get
\beas T(r,F) &\leq& \overline{N}(r,\infty;F)+\overline{N}(r,0;F)+\overline{N}(r,\frac{A}{C};F)+S(r,F)\\
&\leq& \overline{N}(r,\infty;f)+\overline{N}(r,\alpha_{1};f^{m})+\overline{N}(r,\alpha_{2};f^{m})+\overline{N}(r,0;f)\\
&+& \overline{N}(r,\infty;L(f))+\overline{N}(r,\alpha_{1};L(f))+\overline{N}(r,\alpha_{2};L(f))+S(r,f)\\
&\leq& \frac{5}{n}T(r,F)+S(r,F),\eeas
which is a contradiction as $n> 6$.\\
\textbf{Subcase-2.2} $AC=0$\\
\textbf{Subsubcase-2.2.1} $A=0$ and $C\neq0$\\
In this case $B\neq0$ and
$$F=\frac{1}{\gamma G+\delta},$$
where $\gamma=\frac{C}{B}$ and $\delta=\frac{D}{B}$.\\
If $F$ has no $1$-point,  then using the Second Fundamental Theorem and (\ref{pe1.2}), we get
\beas &&T(r,F)\\ &\leq& \overline{N}(r,\infty;F)+\overline{N}(r,0;F)+\overline{N}(r,1;F)+S(r,F)\nonumber\\
&\leq& \overline{N}(r,\infty;f)+\overline{N}(r,\alpha_{1};f^{m})+\overline{N}(r,\alpha_{2};f^{m})+\overline{N}(r,0;f)+S(r,f)\\
&\leq& \frac{3}{n}T(r,F)+S(r,F),\eeas
which is a contradiction as $n> 6$.\par
Thus $\gamma+\delta=1$ and $\gamma\neq0$.\par
So, $$F=\frac{1}{\gamma G+1-\gamma},$$
From above we get $\overline{N}(r,0;G+\frac{1-\gamma}{\gamma})=\overline{N}(r,F)$.\\
If $\gamma\neq1$, then using the Second Fundamental Theorem and (\ref{pe1.2}), we get
\beas &&T(r,G)\\ &\leq& \overline{N}(r,\infty;G)+\overline{N}(r,0;G)+\overline{N}(r,0;G+\frac{1-\gamma}{\gamma})+S(r,G)\nonumber\\
&\leq& \overline{N}(r,\infty;L(f))+\overline{N}(r,\alpha_{1};L(f))+\overline{N}(r,\alpha_{2};L(f))+\overline{N}(r,0;L(f))\\
&+& \overline{N}(r,\infty;f)+\overline{N}(r,\alpha_{1};f^{m})+\overline{N}(r,\alpha_{2};f^{m})+S(r,f)\\
&\leq& \frac{5}{n}T(r,F)+S(r,F),\eeas
which is a contradiction as $n> 6$.\par

Thus $\gamma=1$ and $FG\equiv 1$ which gives
$$f^{mn}(L(f))^{n}=\frac{n^{2}(n-1)^{2}}{a^{2}}(f^{m}-\alpha_{1})(f^{m}-\alpha_{2})(L(f)-\alpha_{1})(L(f)-\alpha_{2})$$
As $n>6$ from the above equation it is clear that $f$ has no pole.\\
Let $z_{0}$ be a $\alpha_{1i}$ point of $f$ of order $s$, where $(\alpha_{1i})^{m}=\alpha_{1}$, then it can't be a pole of $L(f)$ as $f$ has no pole, so $z_{0}$ is a zero of $L(f)$ of order $q$ satisfying $n\leq nq =s$.\par
Clearly from above $\overline{N}(f,\alpha_{1i};f) \leq \frac{1}{n}{N}(f,\alpha_{1i};f)$\par
Similarly, we get $\overline{N}(f,\alpha_{2j};f) \leq \frac{1}{n}{N}(f,\alpha_{2j};f)$\par
Also $\overline{N}(r,\infty;f)=S(r,f)$

Thus by the Second Fundamental Theorem we get
\beas &&(2m-1)T(r,f)\\&\leq& \overline{N}(r,\infty;f)+\sum\limits_{i=1}^{m}\overline{N}(r,\alpha_{1i};f)+\sum\limits_{j=1}^{m}\overline{N}(r,\alpha_{2j};f)+S(r,f)\nonumber\\
&\leq& \frac{2m}{n}T(r,f)+S(r,f)\eeas
which is not possible as $n > 6$.\\
\textbf{Subsubcase-2.2.2} $A\neq0$ and $C=0$\par
In this case $D\neq0$ and
$$F=\lambda G+\mu,$$
where $\lambda=\frac{A}{C}$ and $\mu=\frac{B}{D}$.\par
If $F$ has no $1$ point then similarly as above we get a contradiction.\par
Thus $\lambda+\mu=1$ with $\lambda\neq0$.\par
Clearly $\overline{N}(r,0;G+\frac{1-\lambda}{\lambda})=\overline{N}(r,0;F)$\\
If $\lambda \neq1$, then using the Second Fundamental Theorem and (\ref{pe1.2}), we get
\beas &&T(r,G)\\ &\leq& \overline{N}(r,\infty;G)+\overline{N}(r,0;G)+\overline{N}(r,0;G+\frac{1-\lambda}{\lambda})+S(r,G)\nonumber\\
&\leq& \overline{N}(r,\infty;L(f))+\overline{N}(r,\alpha_{1};L(f))+\overline{N}(r,\alpha_{2};L(f))+\overline{N}(r,0;L(f))\\
&+& \overline{N}(r,0;f)+S(r,f)\\
&\leq& \frac{5}{n}T(r,G)+S(r,G),\eeas
which is a contradiction as $n > 6$.\par
Thus $\lambda=1$ and $F\equiv G$.
So $f^{m}$ and $L(f)$ share $(\infty,\infty)$ and
 \beas && n(n-1)f^{2m}(L(f))^{2}(f^{m(n-2)}-(L(f))^{n-2})-2n(n-2)bf^{m}L(f)\\&&(f^{m(n-1)}-(L(f))^{n-1})+ (n-1)(n-2)b^{2}(f^{mn}-(L(f))^{n})=0\eeas
Substituting $h=\frac{L(f)}{f^{m}}$ we get
\bea\label{pe1.5} &&n(n-1)h^{2}f^{2m}(h^{n-2}-1)-2n(n-2)bhf^{m}(h^{n-1}-1)\\&&+(n-1)(n-2)b^{2}(h^{n}-1)=0.\nonumber \eea
If $h$ is non constant then by lemma \ref{lemma}, we get
\beas &&\{n(n-1)hf^{m}(h^{n-2}-1)-n(n-2)b(h^{n-1}-1)\}^{2}\\&=&-n(n-2)b^{2}(h-1)^{4}\prod\limits_{i=1}^{2n-6}(h-\beta_{i})\eeas
Then by the Second Fundamental Theorem we get
\beas&& (2n-6)T(r,h)\\ &\leq& \overline{N}(r,\infty;h)+\overline{N}(r,0;h)+\sum\limits_{i=1}^{2n-6}\overline{N}(r,0;h-\beta_{i})+S(r,h)\nonumber \\
&\leq& \overline{N}(r,\infty;h)+\overline{N}(r,0;h)+\frac{1}{2}\sum\limits_{i=1}^{2n-6}N(r,0;h-\beta_{i})+S(r,h)\\
&\leq& (n-1) T(r,h)+S(r,h), \eeas
which is a contradiction as $n>6$.\par
Thus $h$ is constant. Hence as $f$ is non-constant and $b\neq0$, we get from equation (\ref{pe1.5}), that $(h^{n-2}-1)=0$, $(h^{n-1}-1)=0$ and $(h^{n}-1)=0$. That is $h=1$. Consequently $f^{m}=L(f).$
\end{proof}


\end{document}